\documentclass[12pt]{article}
\usepackage{amsmath}
\usepackage{amssymb}
\usepackage{url, enumerate}

\newcommand{\eop}{\bigstar}  

\newcommand{\cf}{{\rm cf}}

\newenvironment{proof}{\noindent{\bf Proof.}}{\par\bigskip}

\newtheorem{THEOREM}{Theorem}[section]

\newtheorem{Conclusion}[THEOREM]{Conclusion}

\newtheorem{Hypothesis}[THEOREM]{Hypothesis}

\newtheorem{LEMMA}[THEOREM]{Lemma}

\newtheorem{Main Theorem}[THEOREM]{Main Theorem}
\newenvironment{main Theorem}{\begin{Main Theorem}} 
{\end{Main Theorem}}

\newtheorem{Theorem}[THEOREM]{Theorem}
\newenvironment{theorem}{\begin{Theorem}}{\end{Theorem}}

\newtheorem{Definition}[THEOREM]{Definition}
\newenvironment{definition}{\begin{Definition}}{\end{Definition}}

\newtheorem{Conventions}[THEOREM]{Conventions}

\newtheorem{Main Definition}[THEOREM]{Main Definition}
\newenvironment{main definition}{\begin{Main Definition}}
{\end{Main Definition}}

\newtheorem{Lemma}[THEOREM]{Lemma}
\newenvironment{lemma}{\begin{Lemma}}{\end{Lemma}}

\newtheorem{Notation}[THEOREM]{Notation}

\newtheorem{Convention}[THEOREM]{Convention}

\newtheorem{Note}[THEOREM]{Note}

\newtheorem{Observation}[THEOREM]{Observation}

\newtheorem{Remark}[THEOREM]{Remark}

\newtheorem{Question}[THEOREM]{Question}

\newtheorem{Main Fact}[THEOREM]{Main Fact}
\newenvironment{main Fact}{\begin{Main Fact}}{\end{Main Fact}}

\newtheorem{Fact}[THEOREM]{Fact}

\newtheorem{Subfact}[THEOREM]{Subfact}

\newtheorem{Claim}[THEOREM]{Claim}

\newtheorem{Main Claim}[THEOREM]{Main Claim}
\newenvironment{main claim}{\begin{Main Claim}}{\end{Main Claim}}

\newtheorem{Crucial Claim}[THEOREM]{Crucial Claim}
\newenvironment{crucial claim}{\begin{Crucial Claim}}{\end{Crucial Claim}}

\newtheorem{Subclaim}[THEOREM]{Subclaim}

\newtheorem{Sublemma}[THEOREM]{Sublemma}

\newtheorem{Corollary}[THEOREM]{Corollary}

\newtheorem{Example}[THEOREM]{Example}

\newtheorem{Problem}[THEOREM]{Problem}

\newtheorem{Proposition}[THEOREM]{Proposition}

\newtheorem{Conjecture}[THEOREM]{Conjecture}

\newtheorem{Discussion}[THEOREM]{Discussion}

\newenvironment{Proof of the Subfact}
{\noindent{\bf Proof of the Subfact.}}{\par\bigskip}

\newenvironment{Proof of the Theorem}
{\noindent{\bf Proof of the Theorem.}}{\par\bigskip}

\newenvironment{Proof of the Proposition}
{\noindent{\bf Proof of the Proposition.}}{\par\bigskip}

\newenvironment{Proof of the Conclusion}
{\noindent{\bf Proof of the Conclusion.}}{\par\bigskip}

\newenvironment{Proof of the Observation}
{\noindent{\bf Proof of the Observation.}}{\par\bigskip}

\newenvironment{Proof of the Fact}
{\noindent{\bf Proof of the Fact.}}{\par\bigskip}

\newenvironment{proof of the lemma}
{\noindent{\bf Proof of the Lemma.}}{\par\bigskip}

\newenvironment{Proof of the Claim}
{\noindent{\bf Proof of the Claim.}}{\par\bigskip}

\newenvironment{Proof of the Corollary}
{\noindent{\bf Proof of the Corollary.}}{\par\bigskip}

\newenvironment{Proof of the Subclaim}
{\noindent{\bf Proof of the Subclaim.}}{\par\medskip}

\newenvironment{Proof of the Main Claim}
{\noindent{\bf Proof of the Main Claim.}}{\par\bigskip}

\newenvironment{Proof of the Crucial Claim}
{\noindent{\bf Proof of the Crucial Claim.}}{\par\bigskip}

\newcommand{\rng}{{\rm rng}}

\newcommand{\rest}{\upharpoonright}  

\newcommand{\DD}{{\cal D}}

\newcommand{\FF}{{\cal F}}

\newcommand{\LL}{{\cal L}}

\newcommand{\V}{{\bf V}}

\newcommand{\alg}{\mathfrak A}

\newcount\skewfactor
\def\mathunderaccent#1#2 {\let\theaccent#1\skewfactor#2
\mathpalette\putaccentunder}
\def\putaccentunder#1#2{\oalign{$#1#2$\crcr\hidewidth
\vbox to.2ex{\hbox{$#1\skew\skewfactor\theaccent{}$}\vss}\hidewidth}}
\def\name{\mathunderaccent\tilde-3 }

\author{Mirna D\v zamonja\\ School of Mathematics, University of East Anglia\\Norwich, NR4 7TJ, UK
\\\scriptsize{h020@uea.ac.uk} }

\title{Some Banach spaces added by a Cohen real\\
(to appear in {\em Topology and its Applications})}
\begin{document}
\maketitle
\begin{abstract} We study certain Banach spaces that are added in the extension by one Cohen real.
Specifically, we show that adding just one Cohen real to any model 
adds a Banach space of density $\aleph_1$ which does not embed into any such space in the ground model (Theorem \ref{Cohenreal}). Moreover, such a Banach space can be chosen to be UG
(Theorem \ref{Cohenrealc}). This has consequences on the 
the isomorphic universality number for Banach spaces of density $\aleph_1$, which is hence equal to
$\aleph_2$ in the standard Cohen model and the same is true for UG spaces. Analogous universality results for Banach spaces
are true for other cardinals, by a different proof (Theorem \ref{isomorphisms}(1)).\footnote{The author thanks EPSRC for the grant EP/I00498  and Leverhulme Trust for a Research Fellowship in May 2014-June 2015,  which both supported this research. I sincerely thank IHPST at the Universit\'e Paris 1- Sorbonne for offering me hospitality as an Invited Researcher in the period June 2013-June 2015.
Many thanks to Charles Morgan for sharing his knowledge on morasses and to him and Saharon Shelah for commenting on earlier drafts of this paper.

MSC 2010 Classification: 03E75, 46B26, 46B03, 03C45, 06E15.}
\end{abstract}
\section{Introduction} It is a mixture of sadness and pleasure to be writing an article remembering the great mathematician and friend, Mary Ellen Rudin, whom I have had the pleasure to know for many years. She is missed in many ways but the inspiration she left remains. Writing this paper I chose results that I thought I could share with her, as they fit her taste of using combinatorial set theory in a context that involves topology and analysis. I also hope that by exposing a completely combinatorial way of using the forcing extension by one Cohen real in the context of Banach spaces of weight 
$\aleph_1$, this work will attract the attention of Banach space theorists looking for consistent examples of non-separable Banach spaces constructed without the necessity to use involved set-theoretic techniques.

\section{The extension by one Cohen real and an unembeddable UG space}\label{Cohenlike}
In this section we show two constructions of interesting Banach spaces that are added in the extension by one
Cohen real. 
The results are about Banach spaces but they build on the spirit of a number of earlier results in set theory, namely that adding just one Cohen real causes the universe to contain non-trivial objects of
size $\aleph_1$. The first such result is due to Roitman who showed in \cite{Roitman} that
${\rm MA}(\aleph_1)$ is false in this extension. This result was improved by Shelah in \cite{ShSolovayinaccessible} who proved that there is no Suslin tree in this model, the alternative proofs of which then were given by Todor{\v c}evi\'c \cite{Todorcevicpairs} and Mark Bickford, as given in 
Velleman \cite{Velleman}.

The general technique is to use a neat simplified $(\omega,1)$-morass (constructed in ZFC by Velleman \footnote{who was a student of Mary Ellen Rudin and Ken Kunen}, as explained below) in order to inductively construct a Boolean algebra $\alg$ of size $\aleph_1$ and with properties required by the construction. The point of the morass is that it 
lets us construct our algebra, an object of size $\aleph_1$, by controlling the {\em finite} pieces of
the algebra.
Once $\alg$ is constructed we obtain our space as $C({\rm St}(\alg))$. The reason to use the extension by 
one Cohen real is that that particular forcing interacts very nicely with the morass to let us have a lot freedom in choosing the relations between the generators of $\alg$, which at the end translate into the properties of the Banach space. The idea to extend Cohen-like conditions in incomparable ways are present in the proof by Brech and Koszmider  (\cite{BrKo}, Theorem 3.2) that $l^\infty/c_0$ is not a universal Banach space
for density $\mathfrak c$ in the extension by $\aleph_2$ Cohen reals, and the earlier proof by 
D\v zamonja and Shelah 
(\cite{DjSh614}, Theorem 3.4) that in the model considered there
there is no universal normed vector space over ${\mathbb Q}$ of size 
$\aleph_1$ under isomorphic vector space embeddings. In our case we use further properties of the
Cohen forcing to control objects of size $\aleph_1$
in the extension by controlling the ones that are already in the ground model, see Lemma \ref{guessingCohen}.
 
\begin{theorem}\label{Cohenreal} 
Forcing with one Cohen real adds a Boolean algebra $\alg$ of size $\aleph_1$ such that
in the extension the Banach space $C({\rm St}(\alg))$ does not isomorphically embed into any space $X_\ast$ which is in
the ground model and has density $\aleph_1$.
\end{theorem}

\begin{proof}  Let us recall
that a {\em neat simplified $(\omega,1)$-morass} is a system $\langle \theta_\alpha:\,\alpha\le\omega\rangle, \langle \FF_{\alpha,\beta}:\,\alpha<\beta\le\omega\rangle$ such that
\begin{enumerate}
\item for $\alpha<\omega$, $\theta_\alpha$ is a finite number $>0$, and $\theta_{\omega}=\omega_1$,
\item for $\alpha<\beta\le\omega$, $\FF_{\alpha,\beta}$ is a
set of order preserving functions from $\theta_\alpha$ to $\theta_\beta$,
%
\item 
$\bigcup_{\alpha<\omega}\bigcup_{f\in \FF_{\alpha,\omega}} f^{``}\theta_\alpha=\omega_1$,
\item for all $\alpha<\beta<\gamma\le\omega$ 
we have that $\FF_{\alpha,\gamma}=\{f\circ g:\,g\in \FF_{\alpha,\beta}
\mbox{ and }f\in \FF_{\beta,\gamma}\}$,
\item for all $\alpha<\omega$ we have that $\FF_{\alpha,\alpha+1}=\{{\rm id}_\alpha, h_\alpha\}$ for some $h_\alpha$ such that 
there is a {\em splitting point} $k$ with $h_\alpha\rest k={\rm id}_\alpha\rest k$ and
$h_\alpha(k)> \theta_\alpha$, where
${\rm id}_\alpha$ is the identity function on $\theta_\alpha$,
\item for every $\beta_0,\beta_1<\omega$ and $f_l\in \FF_{\beta_l,\omega}$ for $l<2$ there is $\gamma<\omega$
with $\beta_0,\beta_1<\gamma$, function $g\in \FF_{\gamma,\omega}$ and $f'_l\in \FF_{\beta_l,\gamma}$  such that $f_l
=g\circ f'_l$ for $l<2$.
\end{enumerate}

(Expanded) simplified $(\kappa, 1)$-morasses and their improvement to neat simplified $(\kappa, 1)$-morasses, as defined here for $\kappa=
\omega$, were introduced by Velleman in \cite{Vellemanmorass}, Definition 2.10 and just after Theorem 3.8. respectively. He showed in \cite{VellemanMA} that a neat simplified $(\omega,1)$-morass exist in ZFC. We shall use the terminology {\em neat morass} to refer to 
neat simplified $(\omega, 1)$-morasses.

For future use in the proof we need the following observation which easily follows from the properties of a neat morass:

\begin{lemma}\label{increasing} Suppose that $\langle \theta_\alpha:\,\alpha\le\omega\rangle, \langle \FF_{\alpha,\beta}:\,\alpha<\beta\le\omega\rangle$ is a neat morass. Then:

{\noindent (1)} Suppose that $\alpha\le\beta\le\gamma <\omega$. Then
$\bigcup_{f\in \FF_{\alpha,\gamma}}\rng(f)\subseteq \bigcup_{f\in \FF_{\beta,\gamma}}\rng(f)$.

{\noindent (2)} The sequence $\langle \theta_\alpha\rangle_{\alpha<\omega}$ is strictly increasing and $\theta_\alpha\notin \bigcup_{f\in \FF_{\alpha,\alpha+1}}\rng(f)$.

{\noindent (3)} For any $\alpha, \gamma<\omega$ and $1\le n<\omega$, if $\gamma-\alpha\ge n$ then
$|\theta_\gamma\setminus \bigcup_{f\in \FF_{\alpha,\gamma}}\rng(f)|\ge n$.
\end{lemma}

\begin{proof of the lemma} (1) This follows from item 4. in the definition of a
neat morass, since $\bigcup_{f\in \FF_{\alpha,\gamma}}\rng(f)=
\bigcup_{g\in \FF_{\alpha,\beta}, f\in \FF_{\beta,\gamma}}\rng(f\circ g) \subseteq 
\bigcup_{f\in \FF_{\beta,\gamma}}\rng(f)$.

{\noindent (2)}  These observations follow easily from item 5. in the definition of a
neat morass.

{\noindent (3)} The proof is by induction on $n\ge 1$. The initial stage $\underline{n=1}$ is taken care of
(2), as $\theta_\alpha\notin \bigcup_{f\in \FF_{\alpha,\alpha+1}}\rng(f)$, which suffices by (1). For the induction step
$\underline{n+1}$, let $\beta=\gamma-1$. As before, $\theta_\beta\notin \bigcup_{f\in \FF_{\beta,\gamma}}\rng(f)$, so by (2), $\theta_\beta\notin \bigcup_{f\in \FF_{\alpha,\gamma}}\rng(f)$.

Let $\{s_0, \ldots, s_{n-1}\}$ be a 1-1 enumeration of a subset of $\theta_\beta$ disjoint from
$\bigcup_{f\in \FF_{\alpha,\beta}}\rng(f)$, which exists by the induction hypothesis. We shall show
that for every $s\in \{s_0, \ldots, s_{n-1}\}$ we have that $s\notin \bigcup_{f\in \FF_{\alpha,\gamma}}\rng(f)$, which together with the point $\theta_\beta$ will give a set of $n+1$ elements of 
$\theta_\gamma\setminus \bigcup_{f\in \FF_{\alpha,\gamma}}\rng(f)$. So fix such $s$.

Suppose that $s\in \rng(h)$ for some $h\in  \FF_{\alpha,\gamma}$ and let $g\in \FF_{\alpha,\beta}$
and $g\in \FF_{\beta,\gamma}$ be such that $h=f\circ g$. Let $t\in \rng(g)$ be such that $s=f(t)$.
If $f={\rm id}_\beta$ then $s=t$, so $s\in \rng(g)$, contrary to the choice of the set 
$\{s_0, \ldots, s_{n-1}\}$. If $f=h_\beta$, let $k$ be the splitting point of $h$. If $t<k$ then
again $s=t$, a contradiction as before. If $t\ge k$ then $s=h_\beta(t)>\theta_\beta$, a contradiction.
$\eop_{\ref{increasing}}$
\end{proof of the lemma}

For the rest of the proof, let $\V$ denote the ground model, so any model of ZFC,
and let $\V[G]$ denote the extension of $\V$  by one Cohen real. Unless stated otherwise, 
our arguments take place in $\V$. Fixing a neat morass as in the above definition,
we shall define a Boolean algebra $\alg$ as follows. Let $\LL=\{\le^\ast, d\}$ be a language in which
both $\le^\ast$ and $d$ are binary relation symbols and let $\varphi$ be the $\LL$-sentence
stating that $\le^\ast$ is a partial order, $d$ is symmetric and antireflexive and that for all $x,y$, the statements $d(x,y)$ and $\le^\ast (x,y)$
are contradictory. Our aim is to obtain a model $I$ of $\varphi$ on $\omega_1$, and then interpreting 
$\le^\ast$ as the Boolean $\le$ and $d$ as the disjointness relation in the language of Boolean algebras, generate $\alg$ by $\omega_1$ freely except for the relations in $I$. This is possible by the
compactness theorem and the fact that we have defined $\varphi$ so that it is consistent with the axioms of a Boolean algebra.

To achieve this we first by induction on $\alpha < \omega$ define a model $I_\alpha$ of $\varphi$ on 
$\theta_\alpha$. The basic requirement of the induction will be:

\begin{description}
\item[(i)] if $\alpha< \beta$ and $f\in \FF_{\alpha, \beta}$ then $f$ gives rise to an
$\LL$-embedding from $I_\alpha$to $I_\beta$.
\end{description}

Then we shall define $\le^\ast$ on $\omega_1$ by letting $i \le^\ast j$ iff there are some 
$\alpha< \omega$, $f\in \FF_{\alpha, \omega}$ and $i', j'\in  \theta_\alpha$ such that 
$i'\le^\ast j$ holds in $\theta_\alpha$ and $f(i')=i$, $f(j')=j$. We similarly define $d$. Requirement 6. in the definition of a neat morass and part (i) of the inductive hypotheses give that this constructions defines a well defined model of $\varphi$.

For the main part of the proof,
suppose that in the ground model $\V$ we have a Banach space $X_\ast$  with density $\aleph_1$ and a fixed dense set $\{z_i:\,i<\omega_1\}$ of $X_\ast$. 
We shall guarantee that for all natural numbers $n_*\ge 3$ and for all $j:\,\omega_1\to\omega_1$ there are $i_0, i_1 <\ldots < i_{n_*^2}$ such that:
\begin{equation*}\label{dichotomy}
\begin{split}
\mbox{ if }||z_{j(i_0)}+z_{j(i_1)}+\ldots z_{j(i_{n_*^2})}||<n_*-1&\mbox{, then } i_0, i_1,\ldots i_{n_*^2}\mbox{ are disjoint in }
\alg\mbox{ and }\\
 \mbox{ if }||z_{j(i_0)}+z_{j(i_1)}+\ldots z_{j(i_{n_*^2})}||\ge n_*-1 & \mbox{, then } i_0<_\alg  i_1\ldots <_\alg i_{n_*^2}. 
\end{split}\tag{$\ast$}
\end{equation*}
To see that our algebra, once constructed,
has the required properties, suppose that $T$ is an isomorphic embedding of $C({\rm St}(\alg))$ into some $X_\ast$ as above
and that $(\ast)$ holds for a dense set $\{z_i:\,i<\omega_1\}$ of $X_\ast$. Let
$x_i=T(\chi_{[i]})$ for $i<\omega_1$ and let $n_*\ge 3$ be large enough so that for each $x\in C({\rm St}(\alg))$
we have that 
\[
\dfrac{1}{n_*}||x||< ||T(x)||<n_*||x||,
\]
as is guaranteed to exist by the definition of an isomorphic embedding. For each $i<\omega_1$ let us choose $j(i)$ such that
$||x_i-z_{j(i)}||<\dfrac{1}{n_*^2+1}$. Let $i_0, i_1 <\ldots < i_{n_*^2}$ be as guaranteed by $(\ast)$. Then in the case $||z_{j(i_0)}+z_{j(i_1)}+\ldots z_{j(i_{n_*^2})}||<n_*-1$
we have that 
\[
||x_{i_0}+x_{i_1}+\ldots x_{i_{n_*^2}}||\le ||z_{j(i_0)}+z_{j(i_1)}+\ldots z_{j(i_{n_*^2})}||+\Sigma_{k\le n_*^2} ||x_{i_k}
-z_{j(i_k)}||< n_*-1+ \frac{n^2_*+1}{n^2_*+1}=n_*,
\]
yet $\dfrac{1}{n_*}||\chi_{[i_0]}+\ldots \chi_{[i_{n^2_\ast}]}||=\dfrac{n^2_\ast+1}{n_*}>n_\ast$, in contradiction with the choice
of $n_\ast$. The other case is similar.

Now we claim that to guarantee the condition $(\ast)$ for any fixed $X_\ast$ and $\{z_i:\,i<\omega_1\}$,
it suffices to assure that for all
\[
n_*\ge 3,
A\in [\omega_1]^{\omega_1}\cap
\V , j:\,A\to\omega_1\in \V\mbox{  there are }i_0, i_1 <\ldots < i_{n_*^2}\in A \mbox{ exemplifying }(\ast).
\tag{$\ast\ast$}
\]

For this we shall use the following well known Lemma \ref{guessingCohen}, which is the combinatorial heart of most arguments about the forcing with one Cohen real.  

\begin{lemma}\label{guessingCohen} For every $j:\,\omega_1\to \omega_1$ in $\V[G]$, there is $A\in [\omega_1]^{\omega_1}$ in 
$\V$ and $j_0:\,A\to  omega_1$ in $\V$ such that $j\rest A=j_0$. 
\end{lemma}

\begin{proof of the lemma}
Suppose that $\name{j}$ is a name in the Cohen forcing for a function 
$j:\,\omega_1\to \omega_1$. Then for every $\alpha<\omega_1$ there is $p\in G$ deciding the value 
of $\name{j}(\alpha)$. Since the
forcing notion is countable, there is $p^\ast$ in $G$ such that the set $A$ of all $\alpha$ for which $p$ decides the
value of $\alpha$ is uncountable, and then it suffices to define $j_0$ on $A$ by $j_0(\alpha)=i$ iff $p$ forces
$\name{j}(\alpha)=i$.
$\eop_{\ref{guessingCohen}}$
\end{proof of the lemma}

In order to guarantee the requirement $(\ast\ast)$, we first define by induction on $n$ an increasing sequence of natural numbers
$\alpha_n$ such that
$\alpha_0=0$ and $\alpha_{n+1} \ge \alpha_n +(n+1)$. By Lemma \ref{increasing}(3), for each $n<\omega$ we can find a set $A_n\subseteq \theta_{\alpha_{n+1}}$ of size $n+1$
such that for any $\beta\le \alpha_n$ and $f\in \FF_{\beta,\alpha_{n+1}}$, 
the intersection $A_n\cap {\rm rng}(f)$ is empty.

Our second requirement of the induction will be as follows, where $r$ is
the generic Cohen real, viewed as a function from $\omega$ to 2 :
\begin{description}
\item[(ii)] Let $n<\omega$, suppose that $I_{\alpha_n}$ has been defined and
let $\{i_0, i_1, \ldots ,i_{n}\}$ be the increasing enumeration of $A_n$.
Then if $r(n)=0$ we have that $i_0, i_1,\ldots i_{n}$  are pairwise in the relation $d$  and
if  $r(n)=1$ then $i_0 \le^\ast i_1 \le^\ast \ldots 
\le^\ast i_{n}$.
For $\beta\in (\alpha_n,\alpha_{n+1})$
we let $I_\beta$ be the restriction of $I_{\alpha_{n+1}}$ to $\theta_\beta$.
\end{description}

By the choice of the set $A_n$ it follows that that conditions {\bf (i)} and {\bf (ii)} can be met in a simple inductive construction, as there will be no possible contradiction between {\bf (i)} and the requirements that {\bf (ii)} puts on the elements of
$A_n$.  Let us now show that the resulting algebra $\alg$ is as required. 
Hence let us fix $n_\ast, X_\ast, \{z_i:\,i<\omega_1\}, A$ and $j=j_0$ as in $(\ast\ast)$. Is it is then sufficient to observe that
the following set is dense in the Cohen forcing (note that the set is in the ground model):
\[
\left\{p:\,(\exists n\in {\rm dom}(p))(n \ge  n_*^2\mbox{ and } p(n)=0\iff ||z_{j(i_0)}+z_{j(i_1)}+\ldots z_{j(i_{n_*^2})}||<n_*-1\right\},
\]
where $\{i_0, i_1, \ldots ,i_{n^2_*}\}$ is the increasing enumeration of the first $n_*^2+1$ elements of $A_n$.
$\eop_{\ref{Cohenreal}}$
\end{proof}

The following Theorem \ref{Cohenrealc} shows that the construction from Theorem \ref{Cohenreal} can be refined so that the resulting
Banach space is UG. For completeness we recall the definition of such a space, but we do not define
the G\^ateaux differentiability as we do not need it in the rest of the paper.

\begin{definition} A Banach space $X$ is called {\em UG}
if it admits a uniformly G\^ateaux differentiable renorming.
\end{definition}

A Banach space $C(K)$ of continuous real functions on a compact space $K$ is UG iff $K$ is
a uniform Eberlein compact, as shown by Fabian, Godefroy and Zizler in \cite{FaGoZi}. Our approach to constructing the desired space
$X=C(K)$ will follow the schema described in the introduction to this section and used in the proof of Theorem \ref{Cohenreal}. We shall not recall the definition of a uniformly Eberlein compact space, but only the following notion of a $c$-algebra.

\begin{definition}\label{embeddingBell}
A subset $C$ of a
Boolean algebra $\alg$ has {\em the nice property}
if for no finite $F\subseteq C$ do we have $\bigvee F=1$.
A Boolean algebra $\alg$ is a c-{\em algebra} iff there is
a family $\{ B_n:\,n<\omega\}$ of pairwise disjoint antichains of $\alg$
whose union
has the nice property and generates $\alg$. 
\end{definition}

Bell showed in \cite{bell} that the Stone space of a c-algebra $\alg$ is a uniform Eberlein compact
which therefore implies that for such an algebra the Banach space $C({\rm St}(\alg))$ is an UG  space. We prove:

\begin{theorem}\label{Cohenrealc} 
Forcing with one Cohen real adds a c-algebra $\alg$ of size $\aleph_1$ such that
in the extension the UG Banach space $C({\rm St}(\alg))$ does not isomorphically embed into any
space $X_\ast$ which is in
the ground model and has density $\aleph_1$.
\end{theorem}

To avoid repetitions, we present the proof as a variant of the proof of Theorem\ref{Cohenreal}, to which we refer throughout.

\medskip

\begin{proof}
Fixing a neat morass as in the proof of Theorem \ref{Cohenreal},
we shall define a Boolean algebra $\alg$ similarly as in the proof of that theorem. Let us change the
definition of $\LL$ from the proof of Theorem \ref{Cohenreal} to be a more complex language 
$\LL=\{\le^\ast, d, B_n: n <\omega\}$ in which
both $\le^\ast$ and $d$ are binary relation symbols and all $B_n$ are unary
relation symbols.
Let $T$ be the following set of $\LL$-sentences:
\begin{itemize}
\item $\le^\ast$ is a partial order,
\item $d$ is symmetric and antireflexive,
\item for all $x,y$, the statements $d(x,y)$ and $\le^\ast (x,y)$
are contradictory,
\item for each pair $n\neq m$, for every $x$, $B_n(x)$ and $B_m(x)$ are contradictory,
\item for each $n$, and for every $x\neq y$ with $B_n(x), B_n(y)$ we have $d(x,y)$.
\end{itemize}

A part of our aim is to obtain a model $I$ of $\varphi$ on $\omega_1$, and then interpreting 
$\le^\ast$ as the Boolean $\le$ and $d$ as the disjointness relation in the language of Boolean algebras, generate $\alg$ by $\omega_1$ freely except for the relations in $I$. As in the proof of 
Theorem \ref{Cohenreal}, this is possible by the
compactness theorem since we have defined $T$ so that it is consistent with the axioms for being a set of generators for a Boolean algebra- this can be checked by taking any $c$-algebra as an example. The sentences of $T$
guarantee that $I$ is a disjoint union of $B_n$s and that each $B_n$ is an antichain. However, we shall
have to work a bit harder to obtain that $I$ has the nice property. For this we shall modify the definition of the ordinals $\alpha_n$ so that we shall now inductively choose $\alpha_n<\omega$ so that
$\alpha_{n+1}\ge \alpha_{n} + n + 2$.
By Lemma \ref{increasing}(3) we can fix sets $A_n$ as before, consisting of some
$n+1$ elements of $\theta_{\alpha_{n+1}}\setminus 
\bigcup_{\beta\le\alpha_n,  f \in \FF_{\beta, \alpha_{n+1}} }\rng(f)$ and let also $a_n$ be an element of
$\theta_{\alpha_{n+1}}\setminus  ( \bigcup_{\beta\le\alpha_n,  f \in \FF_{\beta, \alpha_{n+1}}} \rng(f)\cup A_n)$. The requirements of the induction will be the {\bf (i)},
{\bf (ii)} and {\bf (iii)} from the proof of Theorem \ref{Cohenreal}, with the additional requirement 

\begin{description}
\item[(iv)] 
$a_n$ is in the $d$ relation
with every element of $\bigcup_{\beta\le\alpha_n,  f \in \FF_{\beta, \alpha_{n+1}}} \rng(f)$.
\end{description}

To see that it is possible to meet the inductive requirements, for {\bf (i)}-
{\bf (iii)} we use basically the same argument as before, except that we make sure that every element is in some $B_k$ and for {\bf (iv)} we simply choose a large
enough $m$ so that the $B_m$ relation is still empty and we declare $a_n$ a member
of $B_m$. 

It remains to verify that the set $I$ obtained in the limit of the construction has the nice property.
So let $F$ be a finite subset of $I$ and let $\alpha$ be such that for some $h \in \FF_{\alpha,\omega}$
we have that $F\subseteq \rng(h)$. Finding such $\alpha$ is possible by noting that if some
$i <\omega_1$ belongs to $\rng(h)$ for some $h\in \FF_{\alpha,\omega}$ and some $\alpha< \omega$
then for every $\beta \in [\alpha, \omega)$, $i \in \rng(h')$ for some $h'\in \FF_{\beta,\omega}$, by the
property 4. of a neat morass. By a similar reasoning we can assume that $\alpha=\alpha_n$ for
some $n$. Let $H=h^{-1}(F)$. Again applying property 4. of a neat morass, 
we can decompose $h=f\circ g$ for some $g\in \FF_{\alpha_n, \alpha_{n+1}}$ and some
$f \in \FF_{\alpha_{n+1}, \omega}$. Notice that $a_n$ is in the $d$-relation with every element of
$g``(H)$. By the choice of $f$ and the definition of $I$ it follows that $f(a_n)$ is in the $d$ relation
with every element of $F$, hence this positive element $f(a_n)$ in
$\alg$ is disjoint from every element of $F$, so witnessing that $\bigvee F\neq 1$.
$\eop_{\ref{Cohenrealc}}$
\end{proof}

\section{Negative universality results for UG spaces in the iterated Cohen extension and further results on larger densities} 

We shall now explain how the results proved in Section \ref{Cohenlike} have as a consequence that in the standard Cohen model for
violating CH there are no universal Banach spaces, and moreover that this can be witnessed by a UG space. We also show that the same is true at larger densities and in the extensions obtained by Cohen-like forcing. It is a well known theorem of Shelah (see \cite{KjSh409}, Appendix, for a proof) that in the extension obtained by adding a regular 
$\kappa\ge \lambda^{++}$ number of
Cohen subsets to a regular $\lambda$ over a model of GCH, the universality number for models of size $\lambda^+$ for
any unstable complete first order theory is $\kappa=2^\lambda$, so the maximal possible. This in particular applies to Boolean algebras. This result is implied by our results since, as we now explain,
if there is no universal Banach space of a given density $\theta$, then there is no universal
Boolean algebra of size $\theta$. This is the case because standard arguments (see Fact 1.1 in 
\cite{BrKo} imply that every Banach space can be embedded into a Banach space of the same density and the form $C(K)$ for some 0-dimensional space $K$. Combining that with the Stone representation theorem and further standard observations about preservation of isomorphisms, we obtain that the universality number of Banach spaces under isomorphism is never larger than the universality number of Boolean algebras under the isomorphism of Boolean algebras.
 
\subsection{Negative universality results for UG space in the standard Cohen model}\label{standard}
By the standard Cohen model we mean a model obtained by forcing over a model 
${\bf V}$ of $GCH$ to add $\aleph_2$  Cohen reals to obtain a model where $2^{\aleph_0}=
2^{\aleph_1}
=\aleph_2$ (other values of $2^{\aleph_0}$ are possible to obtain in the same way, of course). We claim that the universality number for Banach spaces of density $\aleph_1$ in this model is 
$\aleph_2$, which is the maximal possible value by the above remarks and the fact that in the model
there are only $\aleph_2$ pairwise non-isomorphic Boolean algebras of size $\aleph_1$. Moreover,
this is witnessed by UG spaces (hence their universality number is also  $\aleph_2$), as we shall
show in Theorem \ref{noninvUG}.

\begin{theorem}\label{noninvUG} In the standard Cohen model for $2^{\aleph_0}=\aleph_2$ there is
a family $\FF$ of $\aleph_2$ many UG spaces of density $\aleph_1$ such that no family of $<\aleph_2$
Banach spaces of density $\aleph_1$ suffices to isomorphically embed all members of $\FF$.
\end{theorem}

\begin{proof} Let as stated $\bf V$ be a model of $GCH$ and 
let ${\mathbb P}$ denote the single step Cohen forcing, which we shall for concreteness conceive
as the partial order consisting of finite partial functions from $\omega$ to $2$, ordered by
extension.  Let ${\mathbb Q}$ be the iteration of $\omega_2$ steps of ${\mathbb P}$. Then in the extension by
${\mathbb Q}$ we have that $2^{\aleph_0}=\aleph_2$ and the new reals are added throughout the iteration. Note that by the same considerations mentioned above and the
fact that any $\aleph_1$ Boolean algebras of size $\aleph_1$  can easily be embedded into a single 
Boolean algebra of size $\aleph_1$, the universality number for Banach spaces of density $\aleph_1$
is either 1 or $\aleph_2$, so it suffices to show that there is no single Banach space $X_\ast$ of
density $\aleph_1$ which embeds all UG spaces of density $\aleph_1$.
Suppose for a contradiction that there is such a space $X_\ast$. 

We can by the same argument as above assume that
$X_\ast=C({\rm St}(\alg_\ast))$ for some Boolean algebra $\alg_\ast$ of size
$\aleph_1$. By standard arguments about the iteration of forcing, we can
find an intermediate universe in the iteration, call it ${\bf V'}$ for simplicity, which contains
$\alg_\ast$. If we could say that $C({\rm St}(\alg_\ast))$ is also in ${\bf V'}$, then we could apply Theorem
\ref{Cohenrealc} to conclude that for the generic c-algebra ${\mathfrak B}$ added to ${\bf V'}$ by the 
next iterand in the Cohen iteration, we have that
$C({\rm St}({\mathfrak B}))$ does not embed into $C({\rm St}(\alg_\ast))$ and we would be done. 
Note that by the factoring properties of the Cohen iteration it suffices to work with the case
${\bf V'}=\V$. The problem
is that since we keep adding reals, $C({\rm St}(\alg_\ast))$ keeps changing from ${\bf V}$ to the final universe and
hence there is no contradiction. We resolve this difficulty by a use of simple functions with rational coefficients. 

\begin{lemma}\label{isometries} Let $G$ be a generic for ${\mathbb Q}$.
Then in $\V[G]$ no Banach space
$C({\rm St}(\alg_\ast))$ where $\alg_\ast$ is in $\V$
isomorphically embeds all UG spaces of density $\aleph_1$.\footnote{A preliminary
version of this lemma was stated for isometric embeddings and upon hearing about it, Saharon Shelah suggested that
it should work for the isomorphic embeddings as well.}. 
\end{lemma}

\begin{proof} Suppose that $\alg_\ast$ invalidates the lemma and let $\alg$ be the c-algebra of
size $\aleph_1$ added by ${\mathbb P}$ as in the proof of Theorem \ref{Cohenrealc}, with $K$ its Stone space. Let $p^\ast \in G$ force $\dot{T}$ to be an isomorphic embedding from the $C(K)$ to $X_\ast=C({\rm St}(\alg_\ast))$ with $||T||\le c< n_*$.  
Let $\varepsilon>0$ be small enough, precisely $\varepsilon<\min\left\{\dfrac{n_*-c}{n*},
\dfrac{n^2_*+1-cn_*}{cn_*^2+1}\right\}$. We shall think of the model $I$ obtained in the construction
of $\alg$ in the proof of Theorem \ref{Cohenrealc} as being enumerated as $\{i\,:i<\omega_1\}$, so we
have $\omega_1\subseteq \alg$ and the notation $[i]$ refers to the basic clopen set in $K$ induced
by $i$.
For $i<\omega_1$ let $p_i\ge p^\ast$
force that $h_i$ is a simple function with rational coefficients  (so $h_i\in \V$) satisfying $||\dot{T}(\chi_{[i]})- h_i||<\varepsilon$. Since the Cohen forcing is countable, we can assume by passing to
an uncountable set of conditions that the first coordinate $p_i(0)$ is a fixed condition $r¬\ast$
in the Cohen forcing, and by applying standard arguments using the $\Delta$-system Lemma
we can also assume that for any finite $F\subseteq \omega_1$ the set $\{p_:\, i \in F\}$ has a least common
upper bound $\bigcup_{i\in F} p_{i}$ in ${\mathbb Q}$. 
Let $m$ be such that ${\rm dom}(r^\ast)\subseteq m$. Now applying (a simplified version) of the reasoning 
behind the choice of $\langle \alpha_n:\, n<\omega \rangle $ in the proof of Theorem \ref{Cohenrealc},
we can find a large enough $n>n *^2$ such that $m\le \alpha_n$. Hence among the first elements 
of $A_n$
we can find a set $F=\{i_0, i_1 , \ldots i_{n_*^2}\}$ of size ${n_*^2}+1$ such that $r^\ast$ has two extensions
$r'$ and $r''$ with $r'$ forcing in ${\mathbb P}$ that ${i_k}$s are disjoint in $\alg$
and $r''$ forcing that ${i_k}$s are increasing in $\alg$. Now we can find two extensions
$q'$ and $q'$ of $\bigcup_{i\in F} p_{i}$ in ${\mathbb Q}$ such that $q'(0) \ge r'$ and
$q''(0)\ge r''$.

Without loss of generality, by extending conditions $p_i$ further if necessary,
we can assume that ${\rm dom}(p_i(0))$ is such that for some $\alpha\in {\rm dom}(p_i(0))$ and
some $f_i\in \FF_{\alpha,\omega}$ (using the notation of the proof of Theorem \ref{Cohenrealc})
we have $i\in {\bf \rng}(f_i)$. However the value of $||\Sigma_{k=0}^{n^2_*}h_{i_k}||$ is decided
in $\V$, and if $||\Sigma_{k=0}^{n^2_*}h_{i_k}||\ge n^*$, it follows by the choice of $\varepsilon$ that there cannot
be an extension of $\bigcup_{k=0}^{n^2_*} p_{i_k}$ forcing ${i_k}$s and to be disjoint, a contradiction, and if  $||\Sigma_{k=0}^{n^2_*}h_{i_k}||<n^*$
then it follows that there cannot
be an extension of $\bigcup_{k=0}^{n^2_*} p_{i_k}$ forcing ${i_k}$ to be increasing with $k$, again a contradiction.
$\eop_{\ref{isometries}}$
\end{proof}

$\eop_{\ref{noninvUG}}$
\end{proof}

\subsection{Negative universality results for densities larger than $\aleph_1$}
Let us now consider the situation of adding Cohen subsets to cardinals $\lambda$ possibly larger than
$\omega$.
In this situation we cannot adapt the above techniques to prove that just adding one such subset adds interesting spaces of density  $\lambda^+$. The major difficulties are that the analogues of  neat morass are considerably more complicated and do not just exist in ZFC (although they do exist in the constructible universe ${\bf L}$, for example) and that the Cohen forcing for adding a subset to $\lambda$ no longer satisfies the convenient property given by Lemma \ref{guessingCohen}. However, we can still get negative universality results by generically adding Boolean algebras in an iteration of Cohen-like forcing, as we now show. 

Let $\lambda=\lambda^{<\lambda}$.
Fix a set $A=\{a_i:\,i<\lambda^+\}$ of indices which we shall assume forms an unbounded co-unbounded subset of
$\lambda^+$.

\begin{definition} The forcing  ${\mathbb P}(\lambda)$ consists of Boolean algebras $p$ generated on a subset of $\lambda^+$ by some subset $w_p$ of $A$ of size $<\lambda$
satisfying $p\cap A=w_p$.
The ordering on  ${\mathbb P}$ is given by $p\le q$ if $p$ is embeddable as a subalgebra of $q$ and the embedding fixes
$w_p$, where in our notation $q$ is the stronger condition.
\end{definition}

Let us check some basic properties of the forcing ${\mathbb P}(\lambda)$, reminding the reader of the following notions:

\begin{definition}\label{ccclosed} (1) A forcing notion ${\mathbb P}$ is said to be $\lambda^+$- {\em stationary cc} if
for every set $\{p_i:\,i<\lambda^+\}$ of conditions in ${\mathbb P}$, there is a club $C$ and a regressive function $f$
on $C$ satisfying that for every $i,j<\lambda^+$ of cofinality $\lambda$ satisfying $f(i)=f(j)$, the conditions $p_i$ and $p_j$ are
compatible.

{\noindent (2)} A subset $Q$ of a partial order $P$ is {\em directed} if every two elements of $Q$ have an upper bound in
$P$.  A forcing notion ${\mathbb P}$ is $(<\lambda)$-directed closed if every directed $Q\subseteq {\mathbb P}$
of size $<\lambda$, has an upper bound in ${\mathbb P}$.
\end{definition}

\begin{lemma}\label{properties}
(1)  ${\mathbb P}(\lambda)$ satisfies the $\lambda^+$-stationary cc and is $(<\lambda)$-directed closed.

{\noindent (2)} ${\mathbb P}(\lambda)$ adds a Boolean algebra of size $\lambda^+$ generated by $\{a_i:\,i<\lambda^+\}$.

\end{lemma}

\begin{proof} For part (1), suppose that $\{p_i:\,i<\lambda^+\}$ are conditions in ${\mathbb P}(\lambda)$. Let us denote
by $w_i$ the set $w_{p_i}$ and let us consider it in its increasing enumeration. Let the {\em isomorphism type} of $w_i$
be
determined by the order type of $w_i$ and the Boolean algebra equations satisfied between the elements of $w_i$, denote this by $t(w_i)$. Note that by 
$\lambda^{<\lambda}=\lambda$ the cardinality of the set $\mathcal I$
of isomorphism types is exactly $\lambda$, so let us fix a bijection 
$g:\,{\mathcal I}\times [\lambda^+]^{<\lambda} \to \lambda^+$. 
Note that there is a club $C$ of $\lambda^+\setminus 1$ such that for every point $\gamma$ in $C$ of cofinality $\lambda$ we have
$i<\gamma$ iff $w_i\subseteq \gamma$ and $g``(\,{\mathcal I}\times\gamma)\subseteq\gamma$. Define $f$ on $C$ by letting $f(j)=g(t(w_j), w_j\cap j)$ for $j$ of cofinality $\lambda$ and $0$ otherwise. Hence $f$ is regressive on $C$
and if $i<j$ both in $C$ have cofinality $\lambda$ and satisfy $f(i)=f(j)$ then we have that $w_i\cap i=w_j\cap j= w_i\cap w_j$ and that $p_i$ and $p_j$
satisfy the same equations on this intersection. Hence $p_i$ and $p_j$ are compatible.

For the closure, note that a family of conditions in ${\mathbb P}(\lambda)$ being $(<\lambda)$-directed in particular implies
that the conditions in any finite subfamily agree on the equations involving any common elements. Therefore the union of the
family generates a Boolean algebra which is a common upper bound for the entire family.  

For part (2), note that if $p\in {\mathbb P}$ and $a_i$ is not in $w_p$, then we can extend $p$ to 
$q$ which is freely generated by $w_p \cup\{a_i\}$, except for the equations present in $p$. Hence the set
$\DD_i=\{p:\,a_i\in w_p\}$ is dense. Note that if $p,q\in G$ satisfy
$p\cap A=q\cap A$, then $p$ and $q$ are isomorphic over $w_p$. Let $H$ be a subset of $G$ which is obtained by
taking one representative of each isomorphism class with the induced ordering, hence $H$ is still a filter in ${\mathbb P}$ and it 
intersects every $\DD_i$. By the downward closure of $G$ it follows that for
every $F$ a subset of $A$ of size $<\lambda$, there is exactly one element $p_F$ of $H$ with $w_{p_F}=F$. (Note that no 
new bounded subsets of $\lambda$ are added by ${\mathbb P}$, since (1) holds).
Now note that $\{p_F:\,F\in [A]^{<\lambda}\}$ with the relation of embeddability over $A$ form a directed system of Boolean
algebras directed by $([A]^{<\lambda}, \subseteq)$. Let $B$ be the limit of this system (see \cite{Ha} for an explicit
construction of this object), so $B$ is generated by $\{a_\alpha:\,\alpha<\lambda\}$ and it embeds every element of $H$.
(We shall call this algebra ``the" generic algebra. It is unique up to isomorphism).
$\eop_{\ref{properties}}$
\end{proof}

Note that in the case of $\lambda=\aleph_0$, we in particular have that the forcing is ccc.

\begin{theorem}\label{nonembedding} If $\alg$ is the generic Boolean algebra for some  ${\mathbb P}(\lambda)$, then there is no Banach space $X_\ast$ in the ground model such 
that $C({\rm St}(\alg))$ isomorphically embeds into $X_\ast$ in the extension by ${\mathbb P}(\lambda)$.
\end{theorem}

\begin{proof} Let us fix a $\lambda$.
Suppose for a contradiction that there is an $X_\ast$ and embedding $T$ in the extension contradicting the statement of
the theorem. Let $p^\ast$
force that $\name{T}$ is an isomorphic embedding and, without loss of generality $p^\ast$ decides a positive constant
$c$ such that $1/c \cdot ||x||\le ||Tx||\le c\cdot  ||x||$ for every $x$. Let $n_\ast$ be a positive integer such that $c< n_\ast$.
For each $i<\lambda^+$ let $p_i\ge p^\ast$ be such that 
$a_i\in w_{p_i}$, and without loss of generality $p_i$ decides the value $x_i=T (\chi_{[a_i]})$. Now let us perform
a similar argument as in the proof of the chain condition, and in particular by passing to a subfamily of the
same cardinality we can assume that
the sets $w_{p_i}$ form a $\Delta$-system with a root $w^\ast$  . Since $a_i\in w_{p_i}$, we can assume
that $w_i\neq w^\ast$, and in particular that $a_i\in w_i\setminus w^\ast$. Let us take distinct
$i_0,  i_1, \dots   i_{n^2_\ast}$. Since there
are no equations in $p_i$ or $p_j$ that connect $a_i$ and $a_j$ for $i\neq j$, we can find two conditions $q'$ and $q''$
which both extend $p_{i_0},\ldots p_{n^2_\ast}$ and such that in $q'$ we have $a_{i_0} \le \ldots \le a_{i_{n^2_\ast}}$
and in $q'$ we have that $a_{i_0},\ldots a_{i_{n^2_\ast}}$ are pairwise disjoint.
Hence $q'$ forces that $\chi_{[a_{i_0]}}+\ldots \chi_{[a_{i_{n^2_\ast}}]}$ has norm $n^2_\ast+1$ and
 $q''$ forces that it has norm 1. Therefore the vector $x_{i_0}+\ldots x_{i_{n^2_\ast}}$ must have norm both $>n_\ast$ and 
 $<n_\ast$ in $X_\ast$, which is impossible.
$\eop_{\ref{nonembedding}}$
\end{proof}

We would now like to extend Theorem \ref{nonembedding} 
by using ``classical tricks" with iterations to obtain the non-existence of an isomorphically universal
Banach space of density $\lambda^+$ in an iterated extension, as we did with in the case
of a standard Cohen model. It turns out that
the cases $\lambda=\aleph_0$ which corresponds to adding one Cohen real and larger $\lambda$ are different. In the situation in which we add a generic Boolean algebra of cardinality $\lambda^+\ge\aleph_2$ with
conditions of size $<\lambda$, the forcing is countably closed and therefore it does not add any new reals, making the proof of the analogue of Lemma \ref{isometries} trivial- modulo Lemma 
\ref{nonembedding}. However, to finish the proof we need an iteration theorem for the forcing, in order to make sure that the cardinals are preserved (which in the case of $\lambda=\aleph_0$ is simply a well known theorem about the iteration with finite supports of ccc forcing retaining the ccc property, and hence preserving cardinals). Paper \cite{5authorsforcing} reviews various known forcing axioms and building up on them proves the following Theorem \ref{5authors}, which is perfectly suitable for our purposes.

\begin{theorem} [Cummings, D\v zamonja, Magidor, Morgan and Shelah] \label{5authors} Let $\lambda=\lambda^{<\lambda}$. Then, the iterations with $(<\lambda)$-supports
 of $(<\lambda)$-closed 
stationary $\lambda^+$-cc forcing which is countably parallel-closed, are $(<\lambda)$-closed 
stationary $\lambda^+$-cc.
\end{theorem}

Here, the property of countable parallel-closure is defined as follows:

\begin{definition}\label{parallel-closure} Two increasing sequences $\langle p_i:\,i<\omega\rangle$ and 
$\langle q_i:\,i<\omega\rangle$ of conditions in a forcing $\mathbb P$
are said to be {\em pointwise compatible} if for each $i<\omega$ the conditions $p_i, q_i$ are compatible.
The forcing $\mathbb P$ is said to be {\em countably parallel-closed} if for every two $\omega$-sequences of pointwise
compatible conditions as above, there is a common upper bound to $\{p_i, q_i:\,i<\omega\}$ in $\mathbb P$.
\end{definition}

\begin{lemma}\label{newmet} The forcing ${\mathbb P}(\lambda)$ is countably parallel-closed.
\end{lemma}

\begin{proof} Suppose that $\langle p_i:\,i<\omega\rangle$ and $\langle q_i:\,i<\omega\rangle$ are pairwise compatible increasing sequences. In particular this
means that, on the one hand, each $p_i, q_i$ agree on their intersections, and on the other hand that $\bigcup_{i<\omega} p_i$
and $\bigcup_{i<\omega} q_i$ each form a Boolean algebra. Furthermore $\bigcup_{i<\omega} p_i$ and 
$\bigcup_{i<\omega} q_i$ agree on their intersection, and hence their union can be used to generate a Boolean algebra, which will
then be a common upper bound to the two sequences.
$\eop_{\ref{newmet}}$
\end{proof}

\begin{theorem}\label{isomorphisms} Let  $\lambda=\lambda^{<\lambda}$, let $\cf(\kappa)\ge\lambda^{++}$ and let
$G$ be a generic for the iteration with $(<\lambda)$- supports of length $\kappa$ of the forcing to add the generic Boolean algebra
of size $\lambda^+$ by 
conditions of size $<\lambda$ over a model of GCH. Then in $\V[G]$:
\begin{description}
\item[(1)] 
the universality number of the class of Banach spaces
of density $\lambda^+$ under {\em isomorphisms} is $\kappa=2^{\lambda^+}$.
\item[(2)] 
there are $2^{\lambda^+}$ many pairwise non-isomorphic Banach spaces of density $\lambda^+$.
\end{description}
\end{theorem}

\begin{proof} (1) For the case of $\lambda=\aleph_0$ we use Theorem \ref{isometries}. In the case of uncountable $\lambda$, 
it follows from Lemma  \ref{properties}, Lemma \ref{newmet} and Theorem \ref{5authors} that the
iteration of the  forcing ${\mathbb P}(\lambda)$ described in the statement of the Theorem, preserves cardinals. By the 
countable closure of the forcing no reals are added and hence the analogue of 
Lemma \ref{isometries} holds by Lemma
\ref{nonembedding}, giving the
desired conclusion.

{\noindent (2)} The claim is that the Banach spaces $C({\rm St}(\alg))$ added at the individual steps of the iteration
are not pairwise isomorphic. To prove this, we only need to notice that an embedding of a Banach space into another is determined by the restriction of
an embedding onto a dense set, hence in our case a set of size $\lambda^+$, and then to argue as in (1). 
$\eop_{\ref{isomorphisms}}$
\end{proof}

Using the same reasoning as in Theorem \ref{isomorphisms} with $\lambda=\aleph_0$ and relating this to Theorem
\ref{Cohenrealc} we obtainÉ

\begin{theorem}\label{wcgfinal} In the Cohen model for $2^{\aleph_0}=\kappa$, where ${\rm cf}(\kappa)\ge \aleph_2$, 
the universality number of the UG Banach spaces is at least ${\rm cf}(\kappa)$ and there are ${\rm cf}(\kappa)$ pairwise
non-isomorphic UG Banach spaces.
\end{theorem}

\bibliographystyle{plain}
\bibliography{../biblio}

\end{document}